\documentclass[reqno,11pt]{amsart}

\usepackage{mathrsfs}
\usepackage{graphicx}
\usepackage{amscd}
\usepackage{amsmath}
\usepackage{amsthm}
\usepackage{amsfonts}
\usepackage{amssymb}
\usepackage{amsxtra}
\usepackage{xspace}
\usepackage{array}
\usepackage{cite}
\usepackage{color}
\usepackage{xypic}

\theoremstyle{plain}
\newtheorem{theorem}{Theorem}
\newtheorem{lemma}{Lemma}
\newtheorem{corollary}{Corollary}
\newtheorem{proposition}{Proposition}

\theoremstyle{definition}
\newtheorem{definition}{Definition}

\numberwithin{equation}{section}

\newcommand{\be}{\begin{enumerate}}
\newcommand{\ee}{\end{enumerate}}
\newcommand{\beq}{\begin{equation}}
\newcommand{\eeq}{\end{equation}}
\newcommand{\bprop}{\begin{proposition}}
\newcommand{\eprop}{\end{proposition}}
\newcommand{\complex}{\mathbb{C}}
\newcommand{\reals}{\mathbb{R}}

\newcommand{\integers}{\mathbb{Z}}

\DeclareMathOperator{\ct}{ct}

\newcommand{\s}{\scriptstyle}

\newcommand{\ko}[1][]{K(#1, \,t)}
\newcommand{\kop}[1][]{K^\prime(#1, \,t)}
\newcommand{\ksf}[1][]{K_{#1}}
\newcommand{\kpsf}[1][]{K_{#1}^\prime}
\newcommand{\Wex}{\widehat{W}}
\newcommand{\ttau}{\tau}
\newcommand{\Fd}{\widetilde{F}}
\newcommand{\Fdo}{F_0}
\newcommand{\Ff}{F}
\newcommand{\bcoeff}{b}
\newcommand{\dcoeff}{d}
\newcommand{\ttj}[1][j]{\ttau^{#1}}
\newcommand{\asf}[1][t,q]{a^{\s \Lambda}_{\s \lambda}(#1)} 
\newcommand{\csf}[1][t,q]{c^{\s \Lambda}_{\s \lambda}(#1)} 
\newcommand{\Hp}{\mathcal{H}}
\newcommand{\Hcl}{\Theta}

\newcommand{\cform}[2]{( #1 | #2)}
\newcommand{\anamoly}{s_\Lambda(\lambda)}

\newcommand{\Uplus}{U^+}
\newcommand{\gtil}{\widetilde{G}}
\newcommand{\gotil}{\gtil_0}
\newcommand{\vdag}{v^\dag}
\newcommand{\xdag}{x^\dag}
\newcommand{\ydag}{y^\dag}
\newcommand{\sign}{\mathrm{sign}}
\newcommand{\Hmf}{\theta}

\newcommand{\Lhp}{\mathcal{L}}

\newcommand{\bep}{\bar{\epsilon}}
\newcommand{\kfp}[3][t]{m^{#2}_{#3}(#1)}
\newcommand{\alm}[1][t,q]{a^{\s \Lambda}_{\s \lambda}(#1)}
\newcommand{\Wfin}{\overset{\circ}{W}}
\newcommand{\Qfin}{\overset{\circ}{Q}}
\newcommand{\Pfin}{\overset{\circ}{P}}

\newcommand{\uhp}{\mathbb{H}}
\newcommand{\etat}{\eta^{(-3)}}
\newcommand{\cort}[1][]{\alpha_{#1}^{\scriptscriptstyle\vee}}
\newcommand{\rt}[1][]{\alpha_{#1}}

\newcommand{\kma}{\mathfrak{g}}

\newcommand{\csa}{\mathfrak{h}}

\newcommand{\tsf}{$t$-string function\xspace}
\newcommand{\tsfs}{$t$-string functions\xspace}

\newcommand{\rop}{\Delta_+}

\newcommand{\afsl}{\ensuremath{A_1^{(1)}}}

\begin{document}

\title[]{The $t$-analog of string functions for $A_1^{(1)}$ and Hecke indefinite modular forms.}
\author{Sachin S. Sharma}
\address{Tata Institute of Fundamental Research\\Colaba, Mumbai 400005, India.}
\email{sachin@math.tifr.res.in}
\author{Sankaran Viswanath}
\address{The Institute of Mathematical Sciences\\
Taramani,
Chennai 600113, India.}
\email{svis@imsc.res.in}
\thanks{The second author acknowledges support from DAE under a XII plan project.}

\begin{abstract}
We study the generating functions for Lusztig's $t$-analog of weight multiplicities associated to integrable highest weight representations of the simplest affine Lie algebra $A_1^{(1)}$. These generating series, termed {\em $t$-string functions}, are $t$-analogs of the string functions of $A_1^{(1)}$. String functions are well-studied for all affine Lie algebras and have important modularity properties. However, they are completely known in closed form only for the Lie algebra $A_1^{(1)}$; in this case, Kac and Peterson showed that the string functions can be obtained in terms of certain indefinite modular forms of Hecke. In this paper, we generalize this description to the $t$-string functions of $A_1^{(1)}$. 

\end{abstract}
\maketitle

\section{Introduction}

\subsection{}
Let $\kma$ be an affine Kac-Moody algebra. Let $\Lambda$ be a dominant integral weight of $\kma$, and $L(\Lambda)$ the corresponding irreducible highest weight representation. A weight $\lambda$ of $L(\Lambda)$ is said to be {\em maximal} if $\lambda + \delta$ is not a weight of this module, where $\delta$ is the null root of $\kma$. 
To understand the structure of the module $L(\Lambda)$, one studies the generating functions 
$$ \asf[q] := \sum_{k \geq 0} \dim (L(\Lambda)_{\lambda -  k\delta})\; q^k$$
of weight space dimensions (weight multiplicities) along the $\delta$-string $\{\lambda - k \delta: k \geq 0\}$, where $\lambda$ ranges over the maximal dominant weights of $L(\Lambda)$.
The corresponding {\em string function} $\csf[\tau]$ is defined to be $$\csf[\tau] = q^{\anamoly} \, \asf[q]$$ where 
 $q=e^{2\pi i \tau}$ with $\tau$ in the complex upper half plane, and $\anamoly$ is a rational number (see equation \eqref{anamoly}) which gives the leading power of $q$ in the expansion.  String functions  are modular forms for certain congruence subgroups of $\mathrm{SL}_2(\integers)$ and play an important role in the theory of characters of representations of affine Kac-Moody algebras \cite[chapter 13]{kac}.

Next, we consider Lusztig's $t$-analog of weight multiplicities $\kfp{\Lambda}{\lambda}$ (affine Kostka-Foulkes polynomials). 
Here $\Lambda$ is a dominant integral weight of $\kma$ and $\lambda$ is a dominant weight of $L(\Lambda)$. The $\kfp{\Lambda}{\lambda}$ are polynomials in the indeterminate $t$ which interpolate between the affine Hall-Littlewood functions and the Weyl characters \cite{svis-kfp}.  Further, $\kfp{\Lambda}{\lambda}$ coincides with 
 the Poincar\'{e} polynomial of (the associated graded space of) the affine Brylinski-Kostant filtration on the  
weight space $L(\Lambda)_\lambda$ \cite{slofstra}. Thus they have non-negative integer coefficients, and reduce to the usual weight multiplicities at $t=1$, i.e., 
$\kfp[1]{\Lambda}{\lambda} =  \dim (L(\Lambda)_{\lambda})$. It is therefore natural to consider the generating functions
$$ \asf := \sum_{k \geq 0} \kfp{\Lambda}{\lambda - k\delta}\; q^k$$
for maximal dominant weights $\lambda$ of $L(\Lambda)$. The $t$-analog of the corresponding string function is defined by $\csf[t,\tau] = q^{\anamoly} \, \asf$, with $q= e^{2\pi i \tau}$ as before.

In this paper, we restrict attention to the simplest affine algebra $\kma = A_1^{(1)}$ (affine $\mathfrak{sl}_2$). This is the only case for which an explicit description of all string functions is known \cite{kp-bulletin,kacpeterson}:
\begin{theorem} \rm{(Kac-Peterson)} \label{kpthm}
Let $\kma = A_1^{(1)}$. Let $\Lambda$ be a dominant integral weight of $\kma$, and $\lambda$ a maximal dominant weight of $L(\Lambda)$. Then 
\beq\label{kpthmeq}
\csf[\tau] = \Hmf_L(\tau) \, \eta(\tau)^{\s -3}.
\eeq
 Here $L=L(\Lambda,\lambda)$ is a certain translate of an even rank 2 lattice of signature $(1,1)$, $\Hmf_L(\tau)$ is the corresponding {\em Hecke indefinite modular form}, and $\eta(\tau)$ is the Dedekind eta function.
\end{theorem}
We elaborate on these notions in the next subsection.
We remark here that the first results on string functions of $A_1^{(1)}$ were obtained by Feingold-Lepowsky in \cite{fein-lepow}. They computed the string functions of small level, and in some cases, obtained formulas different from those in the later work of Kac-Peterson. Still different formulas for all of the string functions of $A_1^{(1)}$ were given later by Lepowsky-Primc \cite{lepow-primc}.

The main result of this paper is a description of the $t$-string functions $\csf[t,\tau]$ of $ A_1^{(1)}$ that extends \eqref{kpthmeq}. 
We show that  $\Hmf_L(\tau)$ and $\eta(\tau)^{\s -3}$ in \eqref{kpthmeq} admit 
deformations $\Hcl_L(\tau, t, z)$ and $\etat(\tau, t, z)$ such that $\csf[t, \tau]$ can be obtained as the constant term of the product of these two functions with the Poisson kernel.

\subsection{}
In order to state our main theorem, we recall some background from \cite{kacpeterson}.
Let $\kma = A_1^{(1)}$. Fix a dominant integral weight $\Lambda$ of $\kma$ of level $m \geq 1$, 
and let $\lambda$ be a maximal dominant weight of $L(\Lambda)$.
For $v = (x,y) \in \reals^2$, define
$$ N(v)= 2 (m+2)x^2  - 2m y^2.$$ 
This defines a quadratic form on $\reals^2$.
We also define $\sign(v) = 1$ for $x \geq 0$ and $-1$ for $x<0$. 
Let $M:=\integers^2$ and let $M^*$ denote the lattice dual to $M$ with respect to this form.

Let $O(N)$ denote the group of invertible linear operators on $\reals^2$ 
preserving $N$, and $SO_0(N)$ be the connected component of 
$O(N)$ containing the identity. Consider its subgroups 
$G := \{g \in SO_0(N): g M =M \}$ and $G_0 := \{g \in G: g \text{ fixes } M^*/M \text{ pointwise}\}$.
The set $$\Uplus:=\{(x,y) \in \reals^2: N(x,y) >0\}$$ is preserved under the action of $O(N)$ on $\reals^2$.
We let $x_0:=\frac{\cform{\Lambda + \rho}{\cort[1]}}{2(m+2)}$ and
$y_0:= \frac{\cform{\lambda}{\cort[1]}}{2m}$, where $\cort[1]$ is the simple coroot corresponding to the underlying finite Dynkin diagram of type 
$\mathfrak{sl}_2$, and $\rho$ is the Weyl vector. We have $0 < x_0 < \frac{1}{2}$ and $0 \leq y_0 \leq \frac{1}{2}$.
If $x_0 \geq y_0$, define $(A,B) = (x_0, y_0)$, else $(A,B) = (\frac{1}{2} - x_0, \frac{1}{2} - y_0)$.
Then, $(A,B) \in M^*$, and we set $L= L(\Lambda, \lambda) :=(A,B) + M$.
The Hecke indefinite modular form that occurs in theorem \ref{kpthm} is the following sum over $G_0$-orbits in $L \cap \Uplus$:
\beq\label{heckemfeq}
\Hmf_L(\tau) := \sum_{\substack{v \in\, L \,\cap\, \Uplus \\ v \text{ mod } G_0}} \sign(v) \,   e^{\pi i \tau N(v)}.
\eeq
This is an absolutely convergent sum for $\tau$ in the upper half plane $\uhp$, and defines a cusp form of weight 1 \cite{hecke}.
An equivalent expression for the sum in \eqref{heckemfeq} can be obtained by restricting $v$ to lie in a fundamental domain for the action of $G_0$ on $\Uplus$. The sum then takes the form
$$ \sum_{v  \in D} (\pm) e^{\pi i \tau (N(v) + l(v))},$$
where $l$ is a linear function of $v$ and $D$ is a subset of $\integers^2$ on which $N$ is non-negative.
Such sums appeared in earlier work of Rogers \cite{rogers}, but were first studied systematically by Hecke in \cite{hecke}, where he established their modularity properties. Kac and Peterson used theorem \ref{kpthm}, together with their computation of string functions of low level, to obtain interesting identities 
expressing Hecke modular forms as eta products. Andrews' attempt to understand some of these new identities using  $q$-series methods \cite{andrews1} (see also Bressoud \cite{bressoud}) eventually led him to consider other series of Hecke type, and apply them to the study of Ramanujan's mock theta functions \cite{andrews2, andrews3}.  Since then, 
such series have featured prominently in the vast body of work surrounding the mock theta functions (see for instance \cite{hickerson1, hickerson2, zwegers}).

\subsection{}
Now let $t, z \in \complex$. We first define a deformation $\Hcl_L(\tau, t, z)$ of the Hecke modular form $\Hmf_L(\tau)$ by:
 \begin{equation}\label{hcldef}
\Hcl_L(\tau, t, z) = \sum_{\substack{v \in\, L \,\cap\, \Uplus \\ v \text{ mod } G_0}} \sign(v) \,  e^{\pi i \tau  N(v)}\, \Omega(v, t, z),
\end{equation}
where $\Omega$ is a weighting factor. To define $\Omega$, we need the following subgroup of $O(N)$:
$$\gtil := \langle \zeta \rangle \ltimes G,$$ where $\zeta$ is the reflection about the $Y$-axis, i.e., 
$\zeta(x,y):=(-x,y)$. Then, $\gtil$ contains $G_0$ and 
$$\Fd:=\{(x, y) \in \reals^2: x > 0 \text{ and } 0 \leqslant y \leqslant x\} \cup \{(x,y) \in \reals^2: 0 > y > x\}$$
is a fundamental domain for the action of $\gtil$ on $\Uplus$ (see \S \ref{secthree}). Given $v = (x,y) \in \Uplus$, we define $\vdag = (\xdag,\ydag)$ to 
be the unique element of $\Fd$ which is in the $\gtil$-orbit of $v$. Then, $\Omega$ is defined by
\begin{equation}\label{Omegadef}
\Omega(v, t, z) = t^{2(\ydag-B)} \, z^{\left((m+2) \xdag - m\ydag - \frac{1}{2}\right)}.
\end{equation}
It will turn out that $2(\ydag-B)$ and $(m+2) \xdag - m\ydag - 1/2$ are both integer-valued, piecewise linear functions of $(x,y) \in L$.
The sum in \eqref{hcldef} is absolutely convergent for $\tau \in \uhp, |t| >0, |z|>0$, and satisfies
 $\Hcl_L(\tau,1,1) = \Hmf_L(\tau)$. 

Next, consider the factor involving the Dedekind eta function in \eqref{kpthmeq}:
$$\eta(\tau)^{-3} = e^{-i \pi \tau/4}  \prod_{n=1}^\infty \frac{1}{(1-e^{2 \pi i n \tau})^3}\;\;.$$
We deform this to the function
\beq\label{etatdef}
\etat(\tau, t, z)=e^{-i \pi  \tau/4} \prod_{n=1}^\infty \frac{1}{(1-te^{2 \pi i n \tau})(1-tze^{2 \pi i n \tau}) (1-tz^{-1}e^{2 \pi i n \tau})}.
\eeq
This converges absolutely for $(\tau,t,z)$ in the region
$$\{\tau \in \uhp, \;\; |te^{2 \pi i \tau}| < 1, \;\; |te^{2 \pi i \tau}|< |z|< |te^{2 \pi i \tau}|^{-1}\}$$
and  satisfies $\etat(\tau, 1, 1) = \eta(\tau)^{-3}$. 
Our final ingredient is the function
$$ P(t,z) = \frac{1-t^2}{(1-tz)(1-tz^{-1})} = \sum_{n \in \integers} t^{|n|} z^n.$$
The series converges for $|t|<1$ and $|t|<|z|<|t|^{-1}$. We note that for $t$ real and $|z|=1$, this is just the Poisson kernel of the unit disc. 
Observe that the three functions $\Hcl_L, \etat$ and $P$ have a common domain of definition
$$\{ (\tau,t,z): \tau \in \uhp, \;\; 0 <|t|<1, \;\; |t| < |z| < |t|^{-1}\}.$$ The series/product defining these functions are in fact 
uniformly convergent on compact subsets of this domain. Thus $\Hcl_L, \etat$ and $P$ are holomorphic in each variable $\tau, t, z$.
We view them as functions of $z$ and consider their Laurent series in the annulus  $|t| < |z| < |t|^{-1}$.
We are now in a position to state our main theorem:

\begin{theorem}\label{mainthm} 
Let $\tau \in \uhp, \;0 < |t|<1$. Then, $\csf[t,\tau]$ is the constant term 
{\em (i.e., the coefficient of $z^0$ in the Laurent series in the annulus  $|t| < |z| < |t|^{-1}$)} 
of the product 
$$\Hcl_L(\tau, t, z) \, \etat(\tau, t, z) \, P(t,z).$$
\end{theorem}

The proof of this theorem appears in sections \ref{sectwo} and \ref{secthree}.  In the course of the proof, 
we will derive $t$-analogs of the results of  Kac and Peterson \cite[\S 5]{kacpeterson}.

\subsection{}
Given a holomorphic function $f(z)$ in the annulus $|t| < |z| < |t|^{-1}$, the coefficient of $z^0$ in its Laurent series is just the integral 
$ \int_{0}^1 f(e^{2\pi i u}) \, du.$
Thus, theorem 2 implies that 
$$\csf[t, \tau] =  \int_{0}^1 \Hcl_L(\tau, t, e^{2\pi i u}) \; \etat(\tau, t, e^{2 \pi i u}) \; P(t,e^{2 \pi i u})\, du.$$
Now, if $0<t<1$, then as observed earlier, $P(t,e^{2 \pi i u})$ is the Poisson kernel of the unit disc. The classical fact that the Poisson kernel is an approximate identity now implies that $\lim_{t \to 1^-} \csf[t, \tau] = \Hcl_L(\tau, 1,1) \; \etat(\tau,1,1)$. Thus, theorem \ref{mainthm} reduces to theorem \ref{kpthm} as $t \to 1^-$.

\subsection{}
We point out that a very different constant term expression for the $t$-string function was obtained in \cite[equation (5.8)]{svis-kfp} 
for all affine Lie algebras. Using this and classical summation identities for $q$-series, the \tsfs of levels $1, 2, 4$ 
for $\kma = \afsl$ were obtained in closed form in \cite{sv-a11}.  One can combine the results of \cite{sv-a11} with theorem \ref{mainthm} to obtain $t$-analogs of the Kac-Peterson identities (for levels $1, 2, 4$) relating Hecke modular forms to  eta products. But these would now be constant term identities, and thus not as explicit as in the $t=1$ case.

\medskip
\noindent
{\bf Acknowledgements:} The second author would like to thank E.K.Narayanan for useful discussions regarding this work.

\section{Preliminary Lemmas}\label{sectwo} 
\subsection{} 
We assume throughout that $\kma = A_{1}^{(1)}$, the rank 2 affine Lie algebra with generalized Cartan matrix $\left( \begin{smallmatrix} 2 & -2 \\ -2 & 2 \end{smallmatrix}\right)$. Let $\csa$ denote the Cartan subalgebra of $\kma$.
We let $\rt[0], \rt[1] \in \csa^*$ denote the simple roots, $\delta = \rt[0] + \rt[1]$ the null root, and $\rop = \{\rt[1] + (n-1)\delta, \; -\rt[1] + n\delta, \; n \delta: n \geq 1\}$ the set of positive roots of $\kma$. Let $Q = \integers \rop$ be the root lattice of $\kma$ and $Q_+ := \integers_{\geq 0}\, \rop$. Similarly, let $P$ denote the weight lattice of $\kma$, and $P_+$ the set of dominant integral weights. The standard basis \cite[chapter 6]{kac} of $\csa^*$ is $\{\alpha_1, \delta, \Lambda_0\}$, where $\Lambda_0$ is a fundamental weight corresponding to the extended node of the Dynkin diagram.

Let $W$ denote the Weyl group of $\kma$; this is generated by the two simple reflections $r_0, r_1$. There is a non-degenerate, $W$-invariant, symmetric bilinear form $\cform{\cdot}{\cdot}$ on $\csa^*$ defined by the relations
$\cform{\rt[1]}{\rt[1]} = 2$, $\cform{\Lambda_0}{\delta} =1$ and $\cform{\rt[1]}{\delta} = \cform{\rt[1]}{\Lambda_0} = \cform{\delta}{\delta} = \cform{\Lambda_0}{\Lambda_0} =0 $. Given $\lambda \in \csa^*$, its {\em level} is defined to be $\cform{\lambda}{\delta}$. For $\lambda$ of level $m$, we have 
$$\lambda = \bcoeff(\lambda) \, \alpha_1 + \dcoeff(\lambda) \, \delta + m \Lambda_0$$ where $\bcoeff(\lambda) = \cform{\lambda}{\rt[1]}/2$ and $\dcoeff(\lambda) = \cform{\lambda}{\Lambda_0}$. We also note that if $\lambda \in Q$, then its level is zero, and $\bcoeff(\lambda), \dcoeff(\lambda)$ are both integers.
 Let $\rho$ denote the Weyl vector of $\kma$, defined by the relations 
$\cform{\rho}{\Lambda_0}=0, \cform{\rho}{\rt[1]}=1, \cform{\rho}{\delta}=2$.

Next, recall that the $t$-Kostant partition function $\ko[\beta]$ of $\kma$ is defined by the relation:
\begin{equation}\label{tkodef}
\prod_{\alpha\in \rop} \frac{1}{1-te(-\alpha)} = \sum_{ \beta \in Q_+} \ko[\beta] \,e(-\beta).
\end{equation}
For a general affine Lie algebra, the positive roots must be counted with multiplicities on the product side of \eqref{tkodef}; 
 however all root multiplicities are 1 for $\kma = A_1^{(1)}$.
The $t$-Kostant partition function reduces to the Kostant partition function at $t=1$.

Now, for $\Lambda \in P_+$, let $L(\Lambda)$ denote the corresponding irreducible highest weight representation of $\kma$, 
with weight space decomposition $L(\Lambda) = \oplus_{\gamma \in P} L(\Lambda)_\gamma$. Given $\lambda \in P_+$, 
Lusztig's $t$-analog of weight multiplicity or (affine) Kostka-Foulkes polynomial  $\kfp{\Lambda}{\lambda}$ is defined by 
$$\kfp{\Lambda}{\lambda} := \sum_{w \in W} \epsilon(w) \, \ko[w(\Lambda + \rho) - (\lambda + \rho)].$$ 
Here $\epsilon$ is the sign character of the Weyl group $W$, given by $\epsilon(w) = (-1)^{\ell(w)}$, where
$\ell(w)$ is the minimum number of factors in an expression for $w$ as a product of simple reflections.

By Kostant's weight multiplicity formula, we have 
$\kfp[1]{\Lambda}{\lambda} =  \dim (L(\Lambda)_{\lambda})$.


\subsection{}
We recall the {\em dot action} of $W$ on $\csa^*$. Given $w \in W$ and $\gamma \in \csa^*$, we define  
$$ w \cdot \gamma := w(\gamma + \rho) - \rho.$$
This action leaves $Q$ and $P$ invariant.
Next, following Kac and Peterson \cite{kacpeterson}, we introduce an auxiliary function $\kop[\beta]$ which has better properties than $\ko[\beta]$.
\begin{definition}
Given $\beta \in Q$, define $\kop[\beta]$ as follows:
$$ \kop[\beta] := \ko[\beta] + t \,\ko[r_1 \cdot \beta].$$
\end{definition}

This function reduces to that of Kac and Peterson  \cite[equation (5.1)]{kacpeterson} at $t=1$. 
We also make the following simple observations: if $\beta \in Q$, then (i) $\ko[\beta] =0 \Leftrightarrow \beta \not\in Q_+$, 
and (ii) $\kop[\beta] =0 \Leftrightarrow \beta \not\in Q_+ \cup r_1 \cdot Q_+ \Leftrightarrow \dcoeff(\beta) \leq -1$.

In the following sections, we will study the generating functions of $K, K^\prime$ and $m^{\Lambda}_\lambda$ along $\delta$-strings. 
These are defined for $\beta \in Q$, $\Lambda, \lambda \in P_+$ by:
\begin{align}
\ksf[\beta] &:= \sum_{n \geq 0} \ko[\beta + n\delta] \, q^{n}\label{ksf},\\
\kpsf[\beta] &:= \sum_{n \geq 0} \kop[\beta + n\delta] \, q^{n}\label{kpsf},\\
\alm &:= \sum_{n \geq 0} \kfp{\Lambda}{\lambda-n\delta} \, q^n.
\end{align} 
Here, and in the rest of the paper, we will define $q:=e(-\delta)$ and $z:=e(\alpha_1)$.
The following is our main object of interest:
\begin{definition}
Let $\Lambda \in P_+$ be of level $m \geq 1$, and let $\lambda$ be a maximal dominant weight of $L(\Lambda)$. 
The \tsf $\csf$ is defined by
\beq\label{tsfdef}
\csf := q^{\anamoly} \, \asf,
\eeq
where 
\beq\label{anamoly}
\anamoly = \frac{|\Lambda + \rho|^2}{2(m+2)} - \frac{|\lambda|^2}{2m} - \frac{|\rho|^2}{4}.
\eeq
\end{definition}
Here, we have used the usual notation $|\gamma|^2$ to denote $\cform{\gamma}{\gamma}$.
In our case, this simplifies to \cite{kac,kacpeterson}
$$\anamoly = \dcoeff(\Lambda - \lambda) + \frac{\bcoeff(\Lambda+\rho)^2}{m+2} -  \frac{\bcoeff(\lambda)^2}{m} - \frac{1}{8}.$$

\subsection{}
We recall the definition of the constant term map from \cite{igmgod}.
\begin{definition}
The {\em constant term} map $\ct(\cdot)$ is defined on formal sums $\sum_{\gamma \in Q}
{c_{\gamma}\, e (\gamma)}$ by 
$$\ct\left(\,\sum_{\gamma \in Q} c_{\gamma}\, e (\gamma)\,\right) := \sum_{n \in \integers} c_{n\delta}\, e (n\delta).$$
\end{definition}

It will be convenient to introduce some more notation. 
Let $\Gamma_t$ denote the infinite product in equation \eqref{tkodef}; thus
$$\Gamma_{t} = \prod_{n \geq 1}\left[ \,(1 -t q^n)(1-t q^{n-1} e(-\rt[1]))(1-t q^n e (\rt[1]))\, \right]^{-1}.$$
Next, let  $\xi_t = \Gamma_t \,(1 -  t e(-\rt[1]))$, i.e., 
$$ \xi_t = \prod_{n\geq 1}\left[ \,(1 -t q^n)(1-t q^n e(-\rt[1]))(1-t q^n e (\rt[1]))\,\right]^{-1}.$$
We also define $P_t := \sum_{n\in \integers} t^{|n|} \,e (n\rt[1])$; this can be viewed as  the formal Poisson kernel of the unit disc.
Finally, let $\Lhp:=\{ \beta \in Q: \dcoeff(\beta) \leq 0\}$. 
With these definitions in place, we can now state our first lemma, which describes the generating function of $K^\prime$.
\begin{lemma} \label{one}
If $\beta \in \Lhp$, then 
\begin{equation}\label{kpsfct}
\kpsf[\beta] = \ct\left(\,P_t \; \xi_{t}  \;e (\beta) \right).
\end{equation} 
\end{lemma}

\begin{proof}
Let $\beta \in \Lhp$. Observe that in this case $r_1 \cdot \beta \in \Lhp$, and 
the sums on the right hand sides of equations \eqref{ksf} and \eqref{kpsf} can be replaced by $\sum_{n \in \integers}$. It then follows from definitions that (i) $\kpsf[\beta] = \ksf[\beta] + t \,\ksf[r_1 \cdot \beta]$, (ii)  $\ksf[\beta] = \ct\left(\,\Gamma_t \, e(\beta)\,\right)$ and  (iii)  $\ksf[r_1 \cdot \beta] = \ct\left(\,\Gamma_t \, e(r_1 \cdot \beta)\,\right)$.

\smallskip
For $f = \sum_{\lambda \in \csa^{*}} c_{\lambda}(t)\,e(\lambda)$, define $\overline{f} :=  \sum_{\lambda \in \csa^{*}}
c_{\lambda}(t) \, e(r_1(\lambda))$. Note that $\ct(f) = \ct(\overline{f})$. For $f= \Gamma_t \, e(r_1 \cdot \beta)$, we have $\overline{f} = \overline{\Gamma}_t \, e(\beta + \rt[1])$. Further, it is easy to see that $ \Gamma_t + t\,e(\rt[1]) \, \overline{\Gamma}_t = P_t \; \xi_t$. Putting these facts together, the proof follows.
\end{proof}

\subsection{}
The Weyl group $W$ of $\kma$ can be written as $W = T \rtimes \Wfin$ where $T$ is the group of translations by elements of the finite type root lattice $\Qfin = \integers \rt[1]$, and $\Wfin = \{1, r_1\}$ is the Weyl group of the underlying finite type Dynkin diagram (of type $\mathfrak{sl}_2$ in this case).
The extended affine Weyl group is $\Wex := \widehat{T} \rtimes \Wfin$, where $\widehat{T}$ is the set of translations by elements of the finite type weight lattice $\Pfin = \frac{1}{2} \Qfin$. Letting $\ttau$ denote the translation by the generator $\rt[1]/2$ of $\Pfin$, we have 
$\widehat{T}=\{\ttj[n]:n \in \integers\}$ and $T=\{\ttj[2n]:n \in \integers\}$. 
We also have the following formula for the linear action of $\ttau$ on $\csa^*$ \cite{kac}:
\begin{equation}\label{tteq}
\ttau(\lambda) = \lambda + \frac{1}{2} \left(\,\cform{\lambda}{\delta}\,\rt[1] - \cform{\lambda}{\rt[1]} \,\delta - \cform{\lambda}{\delta}\, \frac{\delta}{2}\,\right).
\end{equation}
We also define the element $\sigma:=\ttau r_1$. The element $\sigma \in \Wex$ permutes the simple roots of $\kma$, i.e., $\sigma(\alpha_0) = \rt[1]$, $\sigma(\rt[1]) = \alpha_0$, and fixes the Weyl vector $\rho$. 

Note that our notation conflicts with the introduction, where $\tau$ was a complex number in the upper half plane. 
But this should cause no confusion, especially since we will henceforth exclusively work in the formal setting, where $q = e(-\delta)$ rather than $e^{2\pi i \tau}$.

\subsection{}
Next, we will obtain an expression for the generating function of $\ksf$. For this purpose, 
define a function $I : Q \times \integers \to \{0, \pm 1\}$ by 
\begin{equation} \label{Idef}
 I(\beta, j) := \begin{cases} 1 & \text{ if } \bcoeff(\beta) \geq 0, j \geq 0. \\
                                        -1 & \text{ if } \bcoeff(\beta) < 0, j < 0. \\
                                         0 & \text{ otherwise}. \end{cases}
\end{equation}
Our second lemma relates the generating functions of $\ksf$ and $\kpsf$, and can be viewed as a $t$-analog of \cite[Theorem C]{kacpeterson}.
\begin{lemma} \label{two}
Let $\beta \in Q$. Then 
\beq\label{koblem}
\ksf[\beta] = \sum_{j \in \integers} (-1)^j \; I(\beta, j) \;
t^j \; \kpsf[\ttj \cdot \beta].
\eeq
\end{lemma}

\begin{proof}
Since $\sigma$ interchanges the simple roots $\alpha_0, \rt[1]$ and fixes $\rho$, it is clear that 
$\ko[\beta] = \ko[\sigma\beta] = \ko[\sigma \cdot \beta]$ for all $\beta \in Q$. Now, this implies that 
$$\ko[\beta] = \kop[\beta] - t\, \ko[r_1 \cdot \beta] = \kop[\beta] - t\, \ko[\ttau \cdot \beta].$$
 Iterating the last expression gives 
\begin{equation} \label{kob}
\ko[\beta] = \sum_{j \geq 0} (-1)^j \, t^j \, \kop[\ttj \cdot \beta].
\end{equation}
Similary, replacing $\beta$ by $\sigma \beta$, one obtains the relations 
$\ko[\beta] = t^{-1} \, \kop[\beta] - t^{-1}\, \ko[\ttau^{-1}\, \cdot \beta]$ and hence
\begin{equation}\label{kopb}
\ko[\beta] = - \sum_{j < 0} (-1)^j \, t^j \, \kop[\ttj \cdot \beta].
\end{equation}
The sums in equations \eqref{kob} and \eqref{kopb} are in fact finite (as can be seen from equation \eqref{taubeta} below) and either expression can be used for a given $\beta \in Q$. But choosing the expression \eqref{kob} (resp. \eqref{kopb}) 
when $\bcoeff(\beta) \geq 0$ (resp. $\bcoeff(\beta) < 0$), we obtain
\begin{equation}\label{kobop}
\ko[\beta] = \sum_{j \in \integers} (-1)^j \, I(\beta,j) \, t^j \, \kop[\ttj \cdot \beta].
\end{equation}
To complete the proof, it only remains to replace $\beta$ by $\beta + n\delta$ ($n \geq 0$) in \eqref{kobop} and observe that  
(i) $ I(\beta + n\delta,j) = I(\beta,j)$ and (ii) $\ttj \cdot (\beta + n\delta) = (\ttj \cdot \beta) + n\delta$.
\end{proof}
We also make the following observation about the ``support'' of the sum in \eqref{koblem}.
\begin{lemma} \label{three} 
Let $\beta \in Q$. Then
\be
\item $\{j \in \integers: I(\beta,j) \neq 0\} \subset \{ j \in \integers: \dcoeff(\ttau^j \cdot \beta) \leq \dcoeff(\beta) \}$,

\smallskip
\item If $I(\beta,j) \neq 0$ and $\beta \in \Lhp$, then  $\ttau^j \cdot \beta \in \Lhp$.
\ee
\end{lemma}

\begin{proof}
The second assertion clearly follows from the first. To prove (1), we use equation \eqref{tteq}
 to obtain:
\begin{equation} \label{taubeta}
\ttj \cdot \beta - \beta = \ttj(\beta + \rho) - (\beta + \rho) = j \rt[1] - \left( j\, \bcoeff(\beta) + \frac{j(j+1)}{2} \right) \, \delta.
\end{equation}
Thus 
$ \dcoeff(\ttj \cdot \beta - \beta) = - j\, \bcoeff(\beta) - \frac{j(j+1)}{2}$. It is clear from equation \eqref{Idef} that this is non-positive for all pairs $(\beta,j)$ for which $I(\beta,j) \neq 0$.
\end{proof}

\subsection{}\label{secllchoice}
Let $\Lambda$ be a dominant integral weight of $\kma$ of level $m \geq 1$, and 
$\lambda$ be a maximal dominant weight of $L(\Lambda)$. Thus, $\cform{\Lambda}{\delta} = m =\cform{\lambda}{\delta}$, and $\lambda + \delta$ is not a weight of $L(\lambda)$. This implies in particular that $\Lambda - \lambda$ is a non-negative integer multiple of $\rt[0]$ or $\rt[1]$. 

Now, it is clear from the symmetry of the Dynkin diagram that $\csf = c^{\s \sigma\Lambda}_{\s \sigma\lambda}(t,q)$. 
If (i) $\Lambda - \lambda \in \integers_{>0} \rt[0]$ or if (ii) $\Lambda = \lambda$ and $\cform{\Lambda}{\rt[1]} > \frac{m}{2}$, we replace 
 $(\Lambda, \lambda)$ by $(\sigma \Lambda, \sigma \lambda)$.
Let  $A:=\frac{\bcoeff(\Lambda + \rho)}{m+2}$ and  $B:= \frac{\bcoeff(\lambda)}{m}$.
It is now easy to see (after this replacement if necessary) that $\Lambda - \lambda \in \integers_{\geq 0} \rt[1]$ and $0 \leq B \leq A < \frac{1}{2}$.
We also observe \cite[pp 259]{kacpeterson} that if $A^\prime=\frac{\bcoeff(\sigma\Lambda + \rho)}{m+2}$ and  
$B^\prime:= \frac{\bcoeff(\sigma\lambda)}{m}$, then $(A^\prime, B^\prime) = (\frac{1}{2}-A, \frac{1}{2} - B)$.

\smallskip
For $w \in W$, define $$s(w) := w(\Lambda + \rho) - (\lambda + \rho) = w \cdot \Lambda - \lambda \in Q.$$
We record the following two elementary facts.
\begin{lemma} \label{four}
$s(w) \in \Lhp$ for all $w \in W$.
\end{lemma}

\begin{proof}
We have $\dcoeff(\beta) = \cform{\beta}{\Lambda_0}$ for all $\beta \in Q$. Thus $\dcoeff(s(w)) = 
\cform{w(\Lambda + \rho) - (\Lambda + \rho)}{ \Lambda_0} + \cform{\Lambda - \lambda}{ \Lambda_0}$. The second term is zero since $\Lambda - \lambda \in \integers \rt[1]$, while the first term equals 
$\cform{\Lambda + \rho}{ w^{-1} \Lambda_0 - \Lambda_0}$ which is non-positive since $\Lambda_0$ is a dominant weight.
\end{proof}

\begin{lemma} \label{five}
$\asf = \sum_{w \in W} (-1)^{\ell(w)} \, \ksf[s(w)].$
\end{lemma}
\begin{proof}
This follows from the definitions.
\end{proof}

\subsection{}
We now put the results of the preceding lemmas together to obtain an expression for the generating function $\asf$.
First, lemmas \ref{two} and \ref{five} imply:
$$ \asf = \sum_{w \in W} \sum_{j \in \integers} (-1)^{\ell(w) + j} \; I(s(w), j) \; t^j \; \kpsf[\ttj \cdot s(w)].$$

Now, from lemmas \ref{three} and \ref{four} it is clear that $\ttj \cdot s(w) \in \Lhp$ for all pairs $(w,j) \in W \times \integers$ for which  
$I(s(w), j) \neq 0$; in this case, we can apply lemma \ref{one} to obtain: 
$$\kpsf[\ttj \cdot s(w)] = \ct\left(P_t \;\xi_t \;e(\ttj \cdot s(w))\right).$$
Now, letting $\bep(w,j) := (-1)^{\ell(w) + j} \,I(s(w), j)$, and 
\beq\label{hpeq}
\Hp:=\sum_{(w,j) \in W\times \integers} \bep(w,j) \; t^j \; e(\ttj \cdot s(w)),
\eeq
we conclude that 
\beq \label{asfeq}
\asf = \ct\left(\,P_t \; \xi_t \; \Hp \,\right).
\eeq

\section{The indefinite quadratic form}\label{secthree}
In this section, we analyze $\Hp$ more closely. 
First, observe that the sum over $(w,j) \in W \times \integers$ in \eqref{hpeq} should really be thought of as a sum over $(w, \ttj) \in W \times \widehat{T}$.
Now, $T$ is a normal subgroup of index 2 in both $W$ and $\widehat{T}$; thus
the group $W \times \widehat{T}$ contains $T \times T$ as a normal subgroup, with the quotient group isomorphic to $\frac{\integers}{2\integers} \oplus \frac{\integers}{2\integers}$. We will show below that on each coset of $T \times T$, the sum defining $\Hp$ has a particularly nice expression  in terms of an indefinite quadratic form on $\integers^2 \cong T \times T$.

\subsection{}
 Let $U:= \reals \rt[1] \oplus \reals \rt[1]$ and $M:= \integers \rt[1] \oplus \integers \rt[1]$. We identify $U$ with $\reals^2$ and $M$ with $\integers^2$. 
 Define a quadratic form $N$ on $U$ by:
$$ N(x,y):= 2 (m+2)x^2  - 2m y^2 \; (x,y \in \reals).$$ 
We observe that $N(\nu)$ is a non-zero even integer for $\nu \in M \backslash \{0\}$.
Let $M^*$ denote the lattice dual to $M$ with respect to the bilinear form induced by $N$. We then have
$$M^* = \frac{1}{2(m+2)} \integers \oplus \frac{1}{2m} \integers.$$ 
Given elements $\mu_1, \mu_2 \in P$ of levels $m+2$ and $m$ respectively, observe that $\left(\frac{\bcoeff(\mu_1)}{m+2}, \frac{\bcoeff(\mu_2)}{m}\right) \in M^*$, since $\bcoeff(\mu_i) = \frac{\cform{\mu_i}{\rt[1]}}{2} \in \frac{1}{2} \integers $. 
\begin{lemma} \label{six}
For $(w,j) \in W \times \integers$, we have 
\beq \label{ttjeq}
\ttj \cdot s(w) = \left((m+2) x - my - \frac{1}{2}\right) \rt[1]  - \frac{1}{2} N(x,y) \, \delta + (\anamoly + \frac{1}{8}) \, \delta \eeq
where $x := \frac{j}{2} + \frac{\bcoeff(w(\Lambda + \rho))}{m+2}$ and  $y := \frac{j}{2} + \frac{\bcoeff(\lambda)}{m}$.
\end{lemma}
\begin{proof}
This is an easy calculation. The coefficient of $\delta$ was computed in \cite[(5.13)]{kacpeterson} for $w \in T$.
\end{proof}
It is now easy to see that equations \eqref{hpeq}, \eqref{asfeq} and \eqref{ttjeq} together imply the following:
\begin{corollary}\label{maincor}
$\csf = \ct\left(P_t \; (q^{-\frac{1}{8}} \xi_t) \; \Hcl\right)$ where 
\begin{equation} \label{hcl}
\Hcl := \sum_{(w,j) \in W\times \integers} \bep(w,j) \, t^j \, 
q^{\frac{1}{2} N(x,y)} z^{(m+2) x - my - 1/2}.
\end{equation}
Here $z:=e(\rt[1])$, and $x, y \in M^*$ are the functions of $(w,j)$ defined in lemma \ref{six}.
\end{corollary}

We remark that $q^{-\frac{1}{8}} \xi_t$ is exactly the function $\etat$ of \eqref{etatdef}.

\subsection{}
We now turn to the map $\phi: W \times \integers \to M^*$ given by $(w,j) \mapsto (x,y)$ where 
$$ x := \frac{j}{2} + \frac{\bcoeff(w(\Lambda + \rho))}{m+2} \text{ and }  y := \frac{j}{2} + \frac{\bcoeff(\lambda)}{m}$$ 
\smallskip
as in lemma \ref{six}.
Letting $e$ denote the identity element of $W$, observe that $\phi(e,0) = (A,B)$ where 
$A=\frac{\bcoeff(\Lambda + \rho)}{m+2}$ and  
$B= \frac{\bcoeff(\lambda)}{m}$ as in \S \ref{secllchoice}. As pointed out there, our choice of $(\Lambda, \lambda)$ ensures
 that $0 \leq B \leq A < \frac{1}{2}$.
The properties of $\phi$ are given by the following lemma.
\begin{lemma} \label{seven}
The map $\phi$ is injective, and its image is a union of translates of $M$. More precisely, we have
Im $\phi = \bigsqcup_{i=1}^4 L_i$, where 
\begin{align}
L_1 &= (A,B) + M, \\ 
L_2 &= \left(A+\frac{1}{2}, B + \frac{1}{2}\right) + M, \\
L_3 &=(-A,B) + M,\\
L_4 &= \left(-A + \frac{1}{2}, B + \frac{1}{2}\right) + M.
\end{align}
\end{lemma}
\begin{proof}
Using lemma \ref{six}, it is clear that $\phi(w_1,j_1) = \phi(w_2,j_2)$ implies $j_1 = j_2$ and 
$\ttj[j_1] \cdot s(w_1) = \ttj[j_2] \cdot s(w_2)$. In turn, this means $s(w_1) = s(w_2)$, and hence $w_1 = w_2$, since $\Lambda + \rho$ is regular dominant. This proves the injectivity.

Next, let $(w,j) \in W \times \integers$. Recall that since $W = T \rtimes \Wfin$, $w$ can be uniquely written as $\ttj[2n]\omega$ for some $n \in \integers, \omega \in \Wfin=\{1,r_1\}$. 
Now, equation \eqref{tteq} implies that  
\beq \label{xyeq}
x = \frac{j}{2} + n + (\mathrm{sgn} \, \omega)A, \;\; y = \frac{j}{2} + B
\eeq
where $\mathrm{sgn}$ is the sign character of $\Wfin$.
It is now clear that if $W \times \integers$ is written as 
the disjoint union of the four subsets 
$$S_1=T \times 2\integers,\;  S_2=T \times (2\integers+1), \; S_3=Tr_1 \times 2\integers, \; S_4=Tr_1 \times (2\integers+1),$$
 then $\phi(S_i) = L_i$ for $1 \leq i \leq 4$.
\end{proof}

We observe that since $0 \leq B \leq A < \frac{1}{2}$, the $L_i$ are pairwise disjoint. 
From lemma \ref{seven}, we see that $\Hcl$ has the following equivalent expression:
\begin{equation} \label{hclxy}
\Hcl = \sum_{(x,y) \in \bigsqcup_{i=1}^4 L_i} \epsilon (x,y)\, t^{2(y-B)} q^{\frac{1}{2} N(x,y)} z^{(m+2) x - my - 1/2},
\end{equation}
where $\epsilon(x,y) :=\bep(\phi^{-1}(x,y))$ for $(x,y) \in \bigsqcup_{i=1}^4 L_i$.

\subsection{}
Let $O(U,N)$ denote the group of invertible linear operators on $U$ preserving the quadratic form $N$, and let $SO_0(U,N)$ be the connected component of 
$O(U,N)$ containing the identity. Let $a \in GL(U)$ be defined by
$$a(u,v) := ((m+1) u + mv, (m+2)u + (m+1)v).$$
Let $G$ be the subgroup of $GL(U)$ generated by $a$, and
 $G_0$ be the subgroup of $G$ generated by $a^2$. It is known that 
$$G = \{g \in SO_0(U,N): g M =M \}$$ 
We note that elements of $G$ also leave $M^*$ invariant, and thus $G$ has a natural action on $M^*/M$. It is known that  $$G_0 = \{g \in G: g \text{ fixes } M^*/M \text{ pointwise}\}.$$
 Define $\zeta \in O(U,N)$ by $\zeta(u,v):=(-u,v)$, and let 
$$ \gtil := \langle \zeta \rangle \ltimes G \text{   and   } \gotil := \langle \zeta \rangle \ltimes G_0.$$ 
We have the following easy properties: (i) $\zeta^2$ is the identity, (ii) $\zeta a \zeta^{-1} = a^{-1}$, (iii) $\gtil$ is an infinite dihedral group.
We have the following diagram of inclusions between the four groups. Each inclusion is as an index 2 subgroup.

$$\xy
\xymatrix{
G_0
\ar[r]
\ar[d]
& G 
\ar[d]\\
\widetilde{G}_0
\ar[r]
& \widetilde{G}
} 
\endxy$$
Observe that $\gtil$ leaves $M$ and $M^*$ invariant; we 
now show that the $\gtil$-orbit of $L_1$ is $\{L_i: 1 \leqslant i \leqslant 4\}$.
\begin{lemma} \label{nine}
(1)  If $g \in G_0$ then $g L_i = L_i$ for  $1 \leqslant i \leqslant 4$.\\
 (2) $L_1 = aL_4 = \zeta L_3 = a \zeta L_2$.
\end{lemma}
 \begin{proof}
The first statement follows from the fact that $G_0$ fixes $M^*/M$ pointwise. To show $L_1 = aL_4$, observe using lemma \ref{six} that $a^{-1}(A,B) = (-A+\frac{1}{2},B+\frac{1}{2}) + \bcoeff(s(e)) (1,-1)$. Since $s(e) = \Lambda - \lambda \in Q$, $\bcoeff(s(e)) \in  \integers$ and we are done.
The remaining two equalities are obvious.
\end{proof}

\subsection{}
We now define the following subsets of $U$: 
\be
\item $\Uplus=\{(u,v): N(u,v) >0\}$,

\smallskip
\item $\Fd = \{(u, v): u > 0 \text{ and } 0 \leqslant v \leqslant u\} \cup \{(u,v): 0 > v > u\}$,

\smallskip
\item $ \Fdo= \Fd \cup a\Fd \cup \zeta \Fd \cup a \zeta \Fd$.
\ee
Clearly, $\Fd \subset \Fdo \subset \Uplus$.

\begin{lemma} \label{ten}
(1) $\Uplus$ is $\gtil$-invariant.\\
(2) $ \Fdo$ is a fundamental domain for the action of $G_0$ on $\Uplus$.

\smallskip
\noindent
(3) $\Fd$ is a fundamental domain for the action of $\gtil$ on $\Uplus$.
\end{lemma}
\begin{proof}
The first assertion follows from the fact that $\gtil \subset O(U,N)$. Now, $\Ff:=\Fd \cup \zeta\Fd$ and $\Fdo = \Ff \cup a\Ff$ are known to be fundamental domains for the actions of $G$ and $G_0$ (respectively) on $\Uplus$ \cite{kacpeterson}. It follows that $\Fd$ is a fundamental domain for the action of $\gtil$ on $\Uplus$.
\end{proof}

The region $\Fd$ arises naturally when one considers the ``support'' of the sum in equation \eqref{hclxy}. More precisely:
\begin{lemma} \label{eight}
For $1 \leq i \leq 4$, we have
$$\{(x,y) \in L_i: \epsilon(x,y) \neq 0\} = L_i \cap \Fd.$$
\end{lemma}
\begin{proof}
We prove this only for $i=1$, the rest of the cases being similar. Fix $(x,y) \in L_1$; by lemma \ref{seven}, we have 
$(x,y) = \phi(w,j)$ where $w = \ttau^{2n}$, $n \in \integers$, $j \in 2 \integers$. Now $\epsilon(x,y) \neq 0$ iff $I(s(w),j) \neq 0$ iff either (i) $n, j \geq 0$ or (ii) $n, j <0$. From equation \eqref{xyeq} and our assumption that $0 \leq B \leq A < \frac{1}{2}$, it follows that (i) is equivalent to $0 \leq y \leq x$ and (ii) is equivalent to $0 > y > x$. 
\end{proof}

Lemmas \ref{nine} and \ref{ten} allow us to identify the sets $\bigsqcup_{i=1}^4 L_i \cap \Fd$ and $L_1 \cap \Fdo$. More precisely, define the map 
$ \psi:  \bigsqcup_{i=1}^4 L_i \cap \Fd \to L_1 \cap \Fdo$ by 
\begin{equation}
\psi(x,y) := \begin{cases} 
(x,y) &\text{ if } (x,y) \in L_1 \cap \Fd\\
a\zeta (x,y) &\text{ if } (x,y) \in L_2 \cap \Fd\\
\zeta (x,y) &\text{ if } (x,y) \in L_3 \cap \Fd\\
a (x,y) &\text{ if } (x,y) \in L_4 \cap \Fd.
\end{cases}
\end{equation}

By lemmas \ref{nine} and \ref{ten}, it is clear that $\psi$ is well-defined, and is a bijection. The inverse map $\psi^{-1}$ is piecewise linear on $ L_1 \cap \Fdo$, and is easy to describe: given $(x,y) \in L_1 \cap \Fdo$, $\psi^{-1}(x,y)$ is the unique element in the $\gtil$-orbit of $(x,y)$ which lies in $\Fd$. 
We will denote $\psi^{-1}(x,y) =(\xdag, \ydag)$.

\subsection{}
We now return to $\Hcl$ in equation \eqref{hclxy} :
$$\Hcl = \sum_{(x,y) \in \bigsqcup_{i=1}^4 L_i} \epsilon (x,y) \,t^{2(y-B)} \,q^{\frac{1}{2} N(x,y)} \,z^{(m+2) x - my - 1/2}$$
where $z:=e(\rt[1])$.
Using lemma \ref{eight}, we can split this into four sums, one over each $L_i \cap \Fd$. We then perform a change of variables, replacing $(x,y) \in \bigsqcup_{i=1}^4 L_i$ by $\psi(x,y) \in  L_1 \cap \Fdo$.
Since $N(x,y) = N(\xdag,\ydag)$, the resulting sum becomes:
$$\Hcl = \sum_{(x,y) \in L_1 \cap \Fdo} \epsilon (\xdag,\ydag)\; t^{2(\ydag-B)} \,q^{\frac{1}{2} N(x,y)}\, z^{(m+2) \xdag - m\ydag - 1/2}.$$

For $(x,y) \in \Uplus$, define $\sign(x,y):= 1$ if $x>0$ and $-1$ if $x <0$. We then have:

\begin{lemma} \label{eleven}
For $(x,y) \in L_1 \cap \Fdo$, $\epsilon (\xdag,\ydag) = \sign(x,y)$.
\end{lemma}
\begin{proof}
As in the above discussion, we split this into the four cases $(x,y) \in L_1 \cap g \Fd$ for (i) $g = e$, (ii) $g = a\zeta$, (iii) $g = \zeta$ and (iv) $g = a$. We only consider case (ii), which is representative of the calculation needed for the other cases.  For $(x,y) \in L_1 \cap a\zeta \Fd$, we have $ (\xdag,\ydag) = (a\zeta)^{-1}(x,y) \in L_2 \cap \Fd$. Let  
$(\xdag,\ydag) = \phi(w,j)$ where $w = \ttau^{2n}$, $n \in \integers$, $j \in 2 \integers+1$. Now $\epsilon(\xdag,\ydag) = \bep(w,j) = - I(s(w),j)$. Now, $\epsilon(\xdag,\ydag)$ equals $-1$ if  $n, j \geq 0$ and $1$ if $n, j <0$. 
In other words $\epsilon(\xdag,\ydag) = - \sign(\xdag,\ydag) = \sign(x,y)$. The last equality follows from the fact that $a$ leaves sign invariant, while $\zeta$ reverses it. 
\end{proof}

Since $\sign(x,y)$ and $N(x,y)$ are constant on $G_0$-orbits, we have :
\begin{equation} \label{hh}
\Hcl = t^{-2B} z^{-\frac{1}{2}} \sum_{\substack{(x,y) \in L_1 \cap \Uplus \\ (x,y) \text{ mod } G_0}} \sign(x,y) \; q^{\frac{1}{2} N(x,y)}\; t^{2\ydag} \, z^{(m+2) \xdag - m\ydag}
\end{equation}
Finally, putting together corollary \ref{maincor}, equation \eqref{hh} and our choice of the pair $(\Lambda, \lambda)$ in \S \ref{secllchoice}, we deduce theorem \ref{mainthm}. \qed

\end{document}